\documentclass[12pt,english,reqno]{amsart}
\usepackage{amssymb,latexsym}
\usepackage{enumerate}

\makeatletter
\@namedef{subjclassname@2010}{%
  \textup{2010} Mathematics Subject Classification}
\makeatother

\newtheorem{theo}{Theorem}[section]

\newtheorem{lemm}[theo]{Lemma}

\theoremstyle{definition}
\newtheorem{defi}[theo]{Definition}
\newtheorem{nota}[theo]{Notation}
\newtheorem{rema}[theo]{Remark}
\newtheorem{exam}[theo]{Example}

\numberwithin{equation}{section}

\usepackage{amsmath, amssymb, amsbsy, amscd, amsthm, epsfig, amsfonts}
\usepackage{amsmath}
\usepackage{array}
\usepackage{hyperref}
\usepackage{amsfonts}
\usepackage[OT2,T1]{fontenc}
\usepackage{cancel}
\usepackage{latexsym}
\usepackage{graphicx}
\usepackage{setspace}
\usepackage[]{fontenc}
\usepackage{geometry}
\usepackage{xypic}
\usepackage{fancyhdr}
\usepackage{enumerate}
\makeindex

\usepackage{babel}
\makeatother

\def\ot{\otimes}

\DeclareMathOperator{\ord}{ord}
\DeclareMathOperator{\Supp}{Supp}
\DeclareMathOperator{\GL}{GL}
\DeclareMathOperator{\rank}{rank}

\DeclareSymbolFont{cyrletters}{OT2}{wncyr}{m}{n}
\DeclareMathSymbol{\Sha}{\mathalpha}{cyrletters}{"58}

\DeclareMathOperator{\Resa}{Resultant}

\def\ol{\overline}

\newfont{\msbm}{msbm10}
\newfont{\msbms}{msbm6}

\def\Q{{\mathbb Q}}
\def\Z{{\mathbb Z}}
\def\al{\alpha}

\def\B{{\mathcal B}}

\def\P{{\mathbb P}}
\def\C{{\mathcal C}}
\def\J{{\mathcal J}}
\def\G{{\mathcal G}}
\def\AA{{\mathcal E}}

\def\om{{\omega}}
\def\F{{\mathbb F}}

\DeclareMathOperator{\Sel}{Sel}
\DeclareMathOperator{\red}{red}

\DeclareMathOperator{\Image}{Image}
\DeclareMathOperator{\Kernel}{Kernel}

\def\th{{\theta}}

\def\p{{\mathfrak p}}
\def\q{{\mathfrak q}}
\def\bw{{\mathbf w}}
\def\bP{{\mathcal{P}}}
\def\bv{{\mathbf v}}
\def\bt{{\mathbf t}}

\def\bn{{\mathbf n}}

\def\bi{\begin{itemize}}
\def\ei{\end{itemize}}
\def\la{\label}

\def\O{{\mathcal O}}

\def\yp{\upsilon}
\def\J{{\mathcal J}}

\newcommand{\be}{\begin{equation}}
\newcommand{\ee}{\end{equation}}
\newcommand{\ben}[1]{\begin{enumerate}[(#1)]}
\newcommand{\een}{\end{enumerate}}
\newcommand{\ba}{\begin{array}}
\newcommand{\ea}{\end{array}}
\newcommand{\bea}{\begin{eqnarray*}}
\newcommand{\eea}{\end{eqnarray*}}

\begin{document}

\baselineskip=17pt

\title{Extending Elliptic Curve Chabauty to Higher Genus Curves}

\thanks{The author is supported by the research grant FCT SFRH/BD/44011/2008. The author would also like to thank Samir Siksek for his valuable comments and guidance}
\author[M. Mourao]{Michael Mourao}
\address{Mathematics Institute, University of Warwick}
\email{M.Mourao@warwick.ac.uk}

\date{}

\begin{abstract} 
We give a generalization of the method of ``Elliptic Curve Chabauty'' to higher genus curves and their Jacobians. This method can sometimes be used in conjunction with covering techniques and a modified version of the Mordell-Weil sieve to provide a complete solution to the problem of determining the set of rational points of an algebraic curve $Y$. 
\end{abstract}

\subjclass[2010]{Primary 11G30 ; Secondary 11G20, 11D41, 11D45}

\maketitle

\tableofcontents{}

\section{Introduction}

The method of Chabauty and Coleman
(\cite{MR0011005},\cite{MR808103}) is a very well
established and explicit technique used to provide
reasonable and sometimes sharp upper bounds for
the set of rational points of a curve defined over $\Q$. To
determine the actual set of rational points, this
is usually used in combination with the
Mordell-Weil sieve (see for example
\cite{MR2041082},\cite{MR2433884},\cite{MR2457355}).
One first splits the analytic set
of $\Q_p$-rational points of the curve, where $p$
is a finite rational prime, into a finite disjoint
collection of 0, or residue classes
as they are often called. Then, Chabauty's argument,
made effective using Coleman's integration on
these rigid analytic spaces (\cite{MR782557}),
often allows one to show that the classes containing
known $\Q$-rational points do not contain any other rational points. The Mordell-Weil sieve is then
used  to prove that the remaining classes (i.e the
ones that appear to have no $\Q$-rational points),
actually have none. The limitation of this method
is the fact that it only applies to curves whose
Jacobians have Mordell-Weil rank less than or
equal to $g-1$, where $g$ is the genus of the
curve.
  
In a more recent development, Siksek
(\cite{siksek}) showed that when Chabauty's method
is generalised to deal with curves defined over a
number field of degree $d>1$
the limitation is usually weakened and the method
can be applied to curves whose Jacobians have
Mordell-Weil rank less than or equal to $d(g-1)$.

Often, when one is interested in the set of
rational points on a curve $Y/\Q$, a descent
argument leads us to consider the following
problem: let $C$ be a curve over a number field $K$.
Let $\psi : C \rightarrow \mathbb{P}^1$ be
a morphism defined over $K$. Determine the
set 
\begin{equation}\label{eqn:subset}
\{P \in C(K) : \psi(P) \in \mathbb{P}^1(\Q)\}.
\end{equation}
For example, 
Bruin (\cite{MR2011330}) 
explains an approach to this when $C$ is a
curve of genus $1$ using a variant of Chabauty called \lq\lq Elliptic
Curve Chabauty\rq\rq. In the present paper we explain an
extension of  Bruin's method to the case where
the curve $C$ has genus greater than $1$.
The fact that we are interested not in all
$K$-rational points of $C$, but in the subset~\eqref{eqn:subset} allows us to
weaken the Chabauty limitation on ranks even
further.

Let $J$ be the Jacobian of $C$.
Our method requires knowledge of a subgroup $L$ of $J(K)$
of finite index. Such a subgroup can sometimes, though
not always, be computed through a descent calculation
(for genus $2$ curves
see \cite{MR1829626};
for cyclic covers of the projective line see \cite{MR1465369}).

In Section \ref{section2} we start by setting up the notation and presenting how the two techniques described in the latter sections can be combined to determine the set of rational points of an algebraic curve.

In Section \ref{chabsection} the modified version of Chabauty is presented explicitly. There is a slight increase in complexity when dealing with points $P_0$ on the curve that are ramification points of the morphism $\psi$, as $\psi-\psi(P_0)$ can no longer be used as a uniformizing parameter in the neighborhoods of these points. This case is thus addressed separately from the case where $\psi$ does not ramify at $P_0$. In the end of the section, we apply the results to three examples of curves defined over a quadratic extension of $\Q$. The outcome of these examples is used in Sections \ref{section4} and \ref{section5}. 

Then, in Section \ref{section4} we show how the classical Mordell-Weil sieve can also be adapted and refined, in order to work together with the version of Chabauty presented in Section \ref{chabsection}.


Finally, in Section \ref{section5} we give an example of a genus $6$ hyperelliptic curve $Y$ defined over $\Q$ by the equation
\[
y^2=(x^3+x^2-1)\Phi_{11}(x),
\]
whose set of $\Q$-rational
points cannot be computed using the classical approach. To apply Bruin's ``Elliptic Curve Chabauty'', one needs to work over the degree $10$ number field $\Q[t]/\left(\Phi_{11}(t)\right)$, and current tools appear to be incapable of computing generators for the Mordell-Weil groups of the associated elliptic curves. We transfer this problem to a collection
of auxiliary genus $2$ curves $\{C_{/K}\}$, for an appropriate
quadratic number field $K$. Even at this step the
rank limiting inequalities given in \cite{siksek}
are not satisfied, but the
inequalities that apply to our case, are. An
implementation of our techniques in
\textsc{Magma}(\cite{MR1484478}) is then used to
successfully prove that 
\[
Y(\Q)=\{\infty\}.
\]

\section{Preliminaries}\la{section2}

\begin{nota}
Let $\O_K$ be the ring of integers of $K$ and let
$\p$ be a prime of $\O_K$. If $A$ is any
$K$-algebra and $u\in A$, we we will denote by
$u^{\p}$, the image of $u$ under the injection
$A\rightarrow A\ot_KK_\p$, where $K_\p$ is the
completion of $K$ at $\p$.
\end{nota}

\begin{defi}
Let $V$ be a non-singular algebraic variety
defined over $K$ and let $p$ be a rational prime,
unramified in $K$ such that $V$ has good reduction
for every $\p\mid p$. Denote the residue field at
$\p$ by $k_\p$. We define the map
$\red_p:V(K)\rightarrow\prod\limits_{\p\mid
p}V(k_{\p})$ to be the diagonal product of the
usual reduction maps  $\red_{\p}:V(K)\rightarrow
V(k_{\p})$. When $V$ is an abelian variety, these
maps are actually homomorphisms of abelian groups.
\end{defi}

\begin{defi}\la{firstdef} Let $C_{/K}$ be a
non-singular algebraic curve. For
$\bP\in\prod\limits_{\p\mid p}C(k_{\p})$ define
\[
\B_p(\bP)=\left\{P\in C(K) :
\red_p(P)=\bP\right\}
\]
and for $P\in C(K)$ define
the $p$-residue class of $P$ to be
\[B_p(P):=\B_p\left(\red_p(P)\right).\]
Fix a morphism $\psi:C\rightarrow \P^1$ defined over $K$. Let 
\[
\G:=\left\{\bP\in \prod\limits_{\p\mid p}C(k_{\p}) : \ol{\psi^\p}(\bP_\p)=\ol{\psi^\q}(\bP_\q)\in\P^1(\F_p)\quad\forall \p,\q\mid p\right\}.
\]
and
\[
H:=C(K)\cap\psi^{-1}\left(\P^1(\Q)\right).
\]
\end{defi}

Consider the following commutative diagram
$$\xymatrix{
H \ar[rr]^{\psi} \ar[dd]_{\red_p} & &\P^1(\Q)\ar[dd]^{\red_p}\\
&&\\
\prod\limits_{\p\mid p}C(k_{\p})\ar[rr]^{\prod\limits_{\p\mid p}\ol{\psi^\p}} && \prod\limits_{\p\mid p}\P^1(\F_p)
}$$

Suppose we have a subset $H'\subseteq H$. In practice $H'$ will be the subset of $H$ found through a computer search. The aim of this paper is to provide methods that often allow one to show that
\begin{enumerate}[(a)]
\item for all $P\in H'$ we have that $B_p(P)\cap H=\{P\}$ (by using a modification of ``Elliptic Curve Chabauty'') and
\item for all $\bP\in \G \setminus\red_p(H')$ we have that $\B_p(\bP)\cap H=\emptyset$ (by using a modification of the Mordell-Weil sieve).
\end{enumerate} 

(a) and (b) put together imply that $H'=H$.

\section{Chabauty}\label{chabsection}

In order to produce the bounds given by Chabauty's method, we first need to use Coleman's theory of $p$-adic integration. Let $C_{/K}$ be an irreducible, non-singular algebraic curve defined over a number field $K$, $p$ be an odd rational prime such that for each $\p$ a prime of $\O_K$ with $\p\mid p$ and residue field $k_\p=\O_K/\p$, we have that $C$ has good reduction at $\p$. Let $J_{/K}$ be the Jacobian variety of $C$. Fix a basepoint $Q\in C(K)$ and denote by $\iota:C\rightarrow J$ the Abel-Jacobi map given by
\[
\iota(P)=[P-Q].
\] 
Denote by $K_\p$ the completion of $K$ with respect to $\p$ and by $\O_\p$ the integers of $K_\p$. Let $\C_{/\O_\p}$ and $\J_{/\O_\p}$ be minimal regular proper models for $C$  and $J$ over $\O_\p$. Denote by $\Omega_{C/K_\p}$ and $\Omega_{J/K_\p}$ the $K_\p$-spaces of global holomorphic $1$-forms on $C$ and $J$. Coleman (\cite{MR782557}) showed that we have a well defined notion of an integral satisfying:
\ben{i}
\item  $\displaystyle\int_{P}^{P'}\omega=-\int_{P'}^{P}\omega$.
\item $\displaystyle\int_{P}^{P'}\omega+\int_{P'}^{P''}\omega=\int_{P}^{P''}\omega$. 
\item $\displaystyle\int_{P}^{P'}\omega+\int_{P}^{P'}\omega'=\int_{P}^{P'}\omega+\omega'$.
\item $\displaystyle\int_{P}^{P'}\al\omega=\al\int_{P}^{P'}\omega$.
\item $\displaystyle\int_{P}^{P'}\iota^*\left(\omega_J\right)=\int_{\iota(P)}^{\iota(P')}\omega_J$.
\een
for $\al\in K_\p$, $P,P',P''\in C(K_\p)$, $\omega,\omega'\in\Omega_{C/K_\p}$, $\omega_J\in\Omega_{J/K_\p}$ and $\iota^*$ the induced isomorphism from $\Omega_{J/K_\p}$ to $\Omega_{C/K_\p}$.

Furthermore, we have the following bilinear pairing
\begin{equation}
\displaystyle\Omega_{C/K_\p}\times J(K_\p)\rightarrow K_\p,\quad \left(\omega,\left[\sum_i P_i'-\sum_i P_i\right]\right)\mapsto\sum_i\int_{P_i}^{P_i'}\omega
\end{equation}
which is $K_\p$-linear on the left with kernel equal to $0$ and $\Z$-linear on the right with kernel the torsion part of $J(K_\p)$.

\subsection{Unramified Case}

Let $C_{/K}$ be a non-singular algebraic curve of genus $g$ and $\psi:C_{/K}\rightarrow\P^1$ a morphism to $\P^1$ which is defined over $K$. Suppose $P_0\in H:=C(K)\cap\psi^{-1}\left(\P^1(\Q)\right)$. Suppose further that $\psi$ is unramified at $P_0$. If $\psi(P_0)=\infty$ replace $\psi$ by $1/\psi$. Now fix a rational prime $p$ such that:

\begin{description}
\item[(p1)] $p$ does not ramify in $K$, i.e. $p\O_K=\p_1\ldots\p_m$ with the $\p_i$ distinct prime ideals.
\item[(p2)] $C_{/K}$ has good reduction at $\p_i$ for $1\leq i\leq m$.
\item[(p3)] The reduced point $\red_{\p_i}(P_0)$ is not a ramification point of the reduced map $\ol{\psi^{\p_i}}$ for $1\leq i\leq m$.
\end{description}

Fix $\p\in\{\p_1,\ldots,\p_m\}$. Let $\C_{/\O_\p}$ be a proper regular minimal model for $C$ over $\O_\p$. Now $\tau:=\psi-\psi(P_0)\in K(C)$ is a uniformizer for $C_{/K}$ at $P_0$. By \textbf{(p1)}, \textbf{(p2)} and \textbf{(p3)} $\tau^\p$ is a well-behaved uniformizer for the generic fibre $\C_{0}$ at $P_0$ and $\ol{\tau^\p}$ is a uniformizer for the special fibre $\C_{\p}$ at $\red_\p(P_0)$.  Equivalently $\tau^\p$ together with $\pi$, a uniformizer of $K_\p$, generate the maximal ideal of the local ring $\O_{C,P_0}$. Let $P\in B_p(P_0)\cap H$. 

Let $\om_1^\p,\ldots,\om_g^\p$ be a basis for the $\O_\p$-module $\Omega_{\C_\p/\O_\p}$ of holomorphic $1$-forms on $\C_{/\O_\p}$. 

Suppose that $L=\langle D_1,\ldots,D_r\rangle$ is a finite index subgroup of the Mordell-Weil group $J(K)$ and that $[J(K):L]=N$. Suppose further that
\begin{equation}\label{decomposition}
N[P-P_0]=n_1'D_1+\ldots +n_r'D_r
\end{equation}
in $J(K)$. Set $n_q=n_q'/N\in\Q$ for $1\leq q\leq r$.
Note that 
\[
t^\p:=\tau^\p(P)=\psi^\p(P)-\psi^\p(P_0)=\left(\psi(P) -\psi(P_0)\right)^\p=\psi(P)-\psi(P_0)\in\Q
\]
for every $\p\mid p$, so
\[
t^{\p_1}=\ldots=t^{\p_m}=:t
\]
and since $\ord_{\p_c}(t^{\p_c})\geq 1$ for $1\leq c\leq m$ we have that $\ord_p(t)\geq 1$. In other words, $t\in p\Z_p$.

We shall need the following standard result.

\begin{lemm}
Let $p$ be an odd rational prime that does not ramify in $K$. Let $\p$ be a prime of $\O_K$ with $\p\mid p$. Fix a minimal regular model $\C_{/\O_\p}$ for $C$ over $\O_\p$. Let $P_0\in C(K_\p)$ and let $\tau\in K(C)$ be a well-behaved uniformizer at $P_0$. Let $\om\in\Omega_{\C/\O_\p}$, and write
\[
\al=\left.\frac{\omega}{d\tau}\right|_{\tau=0}.
\]
Then $\al\in\O_\p$. Moreover, for all $P\in B_p(P_0)$,
\be\la{basisexp}
\int_{P_0}^P\om=\al\tau(P)+\beta\tau(P)^2
\ee
for some $\beta\in\O_\p$ (which depends on $P$).
\end{lemm}

\begin{proof}
For a proof of this see \cite{siksek} Lemma 3.1
\end{proof}

Let $\p\mid p$ and $\om\in\Omega_{\C/\O_\p}$ be a holomorphic $1$-form. We will define the matrix $A_{\p,\om}\in M_{d_\p,r}(\Q_p)$ and the column vector $\bw_{\p,\om}\in\Z_p^{d_\p}$, where $d_\p=[K_\p:\Q_p]$, as follows:\\
Using \eqref{decomposition} and \eqref{basisexp} together we get the following equality in $K_\p$
\be\la{expand}
\al_1n_1+\ldots+\al_rn_r=\al t+\beta t^2,
\ee 
where $\al_q=\int_{D_q}\om$ for $1\leq q \leq r$. Fix an integral basis $\th_1 ,\ldots,\th_{d_\p}$ for $\O_\p$ over $\Z_p$. We can write $\al_q=a_{1,q}\th_1 + \ldots+a_{d_\p,q}\th_{d_\p}$, $\al=a_1\th_1 +\ldots+a_{d_\p}\th_{d_\p}$ and $\beta=b_1\th_1 +\ldots+b_{d_\p}\th_{d_\p}$ and equate coefficients to get the following system of equations in $\Q_p$ 
\begin{eqnarray*}
a_{1,1}n_1+\ldots+a_{1,r}n_r &=& a_1t +  b_1t^2\\ 
\vdots &=& \vdots \\ 
a_{d_\p,1}n_1+\ldots+a_{d_\p,r}n_r &=& a_{d_\p}t + b_{d_\p}t^2.
\end{eqnarray*}
Define $A_{\p,\om}=\left(a_{c,q}\right)_{1\leq c\leq d_\p,1\leq q \leq r}$ and $\bw_{\p,\om}$ to be the column vector $(a_1,\ldots,a_{d_\p}) \in \Z_p^{d_\p}$.

Now define the matrix $A_\p\in M_{gd_\p,r}(\Q_p)$ and the column vector $\bw_{\p}\in\Z_p^{gd_p}$ as 
\[
A_\p=
\left(\ba{c} 
A_{\p,\om_1^\p}\\ 
\vdots \\ 
A_{\p,\om_g^\p}
\ea\right)
\quad\text{and}\quad\bw_{\p}=
\left(\ba{c}
\bw_{\p,\om_1^\p}\\ 
\vdots \\ 
\bw_{\p,\om_g^\p}
\ea\right)
.
\] 
Finally define the matrix $A\in M_{dg,r}(\Q_p)$ and the column vector $\bw\in\Z_p^{gd}$ as
\[
A=
\left(\ba{c} 
A_{\p_1}\\ 
\vdots \\ 
A_{\p_m}
\ea \right)
\quad\text{and}\quad\bw =
\left(\ba{c}
\bw_{\p_1} \\ 
\vdots \\ 
\bw_{\p_m}\ea\right)
.
\]
We now have 
\[
A \bn =t\bw +t^2\bw',
\]
where $\bn=(n_1,\ldots,n_r)\in\Q^r$ and $\bw'\in\Z_p^{gd}$. Let $h$ be the smallest non-negative integer such that $p^hA$ has entries in $\Z_p$ and $U$ be the unimodular matrix in $M_{gd,gd}(\Z_p)$ such that 
\[
A'=U\left(p^hA\right)
\]
is in Hermite Normal Form. Let $j$ be the number of zero rows of $A'$ and denote by $\AA_p(P_0)$ the set containing only the column vector in $\Z_p^j$ formed by the last $j$ rows of $U\bw$. The reason for defining the set $\AA_p(P_0)$ that only contains a single element will become apparent when we discuss how we deal with the case of $\psi$ being ramified at $P_0$ in \ref{ssramified}.

\begin{theo}\la{firsttheo}
If the unique $E\in\AA_p(P_0)$ satisfies $\ol E\neq \mathbf 0$ modulo $p$, then $B_p(P_0)\cap H=\{P_0\}$. 
\end{theo}

\begin{proof}
Let $P\in B_p(P_0)$ with $P\neq P_0$ and $\psi(P)\in\P^1(\Q)$. Using the fact that $\psi-\psi(P_0)$ is a local isomorphism we see that 
\[
s=\ord_p(t)
\]
is finite. We have that for the unique $E\in \AA_p(P_0)$
\[
tE+t^2\bw''=\mathbf{0},
\]
where $\bw''$ is the vector formed by the last $j$ rows of $U\bw'$. Now divide by $p^s$ and reduce modulo $p$ to get that 
\[
v\ol E\equiv\mathbf 0\mod p.
\]
But this is a contradiction since both $v$ and $\ol E$ are non-zero modulo $p$.
\end{proof}

\subsection{Ramified Case}\la{ssramified}

Suppose we have $P_0\in H'$ such that $\psi$ ramifies at $P_0$. Write $e_\psi(P_0)=e$ for the ramification index of $\psi$ at $P_0$. Then $e\geq 2$. Define the modified properties:
\begin{description}
\item[(p1)$^{\text{split}}$]$p\O_K=\p_1\ldots\p_d$, in other words $p$ splits completely in $\O_K$, and $\gcd(p,e)=1$.
\item[(p1)$^{\text{inert}}$]$p\O_K=\p$, in other words $p$ is inert in $\O_K$.
\item[(p3)$^{\text{ram}}$]$e_{\ol{\psi^\p}}\left(\red_\p(P_0)\right)=e$ for every prime $\p\mid p$
\end{description}
If $[K:\Q]=2$ then choose an odd rational prime $p$ that satisfies either \textbf{(p1)$^{\text{split}}$,(p2)} and \textbf{(p3)$^{\text{ram}}$} or \textbf{(p1)$^{\text{inert}}$,(p2)} and \textbf{(p3)$^{\text{ram}}$}, otherwise choose $p$ such that it satisfies \textbf{(p1)$^{\text{split}}$,(p2)} and \textbf{(p3)$^{\text{ram}}$}.

\subsubsection*{$p$ splits} Let us first consider the case where $K$ is a number field of degree $d>1$ over $\Q$ with ring of integers $\O_K$ and that $p$ is an odd rational prime that splits completely over $K$. Equivalently $p\O_K=\p_1\ldots\p_d$ with each $\p_c$, $1\leq c\leq d$, a prime of $\O_K$ of norm equal to $p$. To simplify the notation let $K_c$ denote the completion of $K$ with respect to $\p_c$ and $\O_c$ be the ring of integers of $K_c$. Also let $\C_{/\O_c}$ be a minimal, regular and proper model for $C$ over $\O_c$.
\begin{lemm}\la{lemma1}
Suppose that $P_0\in H$ has $e_{\psi}(P_0)=e\geq 2$ and let $p$ be a prime satisfying \textbf{(p1)$^{\text{split}}$,(p2)} and \textbf{(p3)$^{\text{ram}}$}. Let $\tau_c$ be a well-behaved uniformizer of $\C_{/\O_c}$ at $P_0$ and denote by 
\[
\yp_cT_c^e+\sum_{i=1}^\infty \rho_{c,i} T_c^{e+i}\in K_{\p_c}[[T_c]]=\Q_p[[T_c]]
\]
the formal powerseries expansion of $\psi^{\p_c}-\psi(P_0)$ in terms of $\tau_c$. Suppose there exists $P\in\left(B_p(P_0)\cap H\right)\setminus\{P_0\}$. Then for every $c\in\{2,\ldots,d\}$, $\yp_1T_1^e-\yp_cT_c^e\in\Z_p[T_1,T_c]$ has a linear factor $T_1-\hat\gamma_cT_c$ satisfying
\[
t_1-\hat\gamma_c t_c\equiv 0 \mod p^{s(e+1)},
\]
where $t_c:=\tau_c(P)$ for $1\leq c\leq d$ and $s=\ord_p(t_1)=\ord_p(t_c)\geq 1$. 
\end{lemm}

\begin{proof}
By substituting $P$ in the powerseries expansion of $\psi^{\p_c}-\psi(P_0)$ we get the $d$ equations
\[
\psi(P)^{\p_c}-\psi(P_0)=\yp_ct_c^e+\rho_ct_c^{e+1}
\]
where $\yp_c\in\Z_p^*$, since $p$ satisfies \textbf{(p3)$^{\text{ram}}$}, and $\rho_c\in\Z_p$. Since $\psi(P)\in\Q$ we have that
\[
\psi(P)^{\p_1}=\ldots=\psi(P)^{\p_d}=\psi(P)\in\Q.
\]
In particular we have that
\[
\ord_p\left(t_1\right)=\ldots=\ord_p\left(t_d\right),
\]
since $\yp_c\in\Z_p^*$ for every $c\in\{1,\ldots,d\}$. Let us denote this positive integer by $s$. We have the following $d-1$ congruences
\[
\yp_1t_1^e-\yp_ct_c^e\equiv 0\mod p^{s(e+1)},
\]
for $c\in\{2,\ldots,d\}$. By letting $\gamma_c=\frac{t_1}{t_c}$ we have that $\gamma_c$ is a solution to
\[
\yp_1X^e-\yp_c\equiv 0\mod p^{se}.
\]
Since the derivative of this polynomial is equal to $eX^{e-1}$ and $e$ and $\gamma_c$ are units modulo $p$ we can use Hensel's Lemma to lift $\gamma_c$ to a solution $\hat\gamma_c\in\Z_p^*$. We then have that
\[
(X-\hat\gamma_c)\mid (\yp_1X^e-\yp_c)\quad \text{and} \quad(T_1-\hat\gamma_cT_c) \mid(\yp_1T_1^e-\yp_cT_c^e).
\]
Furthermore we have that
\[
\gamma_c\equiv\hat\gamma_c\mod p^{se}
\]
which implies that
\[
t_1-\hat\gamma_c t_c\equiv 0\mod p^{s(e+1)}.
\]  
\end{proof}

Let $P_0,e,p,t_1,\ldots,t_d,\yp_1,\ldots,\yp_d$ be as in the statement of Lemma \ref{lemma1} above. Suppose that $(X-\hat\gamma_c^{(1)}),\ldots,(X-\hat\gamma_c^{(l_c)})$ are all the linear factors of $\yp_1X^e-\yp_c$ for $c\in\{2,\ldots,d\}$. Define the matrices $E_{(i_2,....,i_d)}\in M_{d-1,d}(\Z_p)$ by
\[
E_{(i_2,....,i_d)}=
\left(\ba{cccc} 
1 & -\hat\gamma_2^{(i_2)} & \ldots & 0 \\ 
\vdots & \vdots & \vdots & \vdots \\ 
1 & 0 & \ldots & -\hat\gamma_d^{(i_d)} 
\ea\right)
\] 
for $i_c\in\{1,\ldots,l_c\}$. 

Let $\{\om_f^c\}_{1\leq f \leq g}$ be a basis of $\Omega_{\C/\O_c}$ for $1\leq c\leq d$ and let $L_0=\langle D_1,\ldots,D_r\rangle$ be a subgroup of $J(K)$ of index $N\in\Z_{>0}$. Now fix a $c\in\{1,\ldots,d\}$. Define $A_c$ to be the $g\times r$ matrix with entries in $\Q_p$ defined by $A_c=\left(a_{f,q}\right)_{1\leq f\leq g,1\leq q\leq r}$ where
\[
a_{f,q}=\int_{D_q}\om_f^c
\]
and by $\bw_c$ the $g\times d$ matrix with zero entries everywhere apart from the $c$-th column which will consist of the vector $(a_1,\ldots,a_g)$ where $a_f$ is the coefficient of the linear term in the (formal) powerseries expansion of 
\[
\int_{[P-P_0]}\om_f^c
\] 
in terms of $\tau_c$ (the uniformizer of $\C_{/\O_c}$ at $P_0$). Now let $A$ be the $dg\times r$ matrix 
\[
A=
\left(\ba{c}
A_1 \\ 
\vdots \\
A_d
\ea\right)
\] 
with entries in $\Q_p$ and $\bw$ be the $gd\times d$ matrix
\[
\bw=
\left(\ba{c}
\bw_1 \\ 
\vdots \\ 
\bw_d
\ea\right)
\]
with entries in $\Z_p$. Let $h$ be the smallest non-negative integer such that $p^hA$ has entries in $\Z_p$ and $U\in \GL_{dg}(\Z_p)$ be the unimodular matrix such that
\[
A'=U(p^hA)
\]
is in Hermite Normal Form. We denote by $E_0\in M_{j,d}(\Z_p)$ the matrix formed by the last $j$ rows of $U\bw$, where $j$ is the number of zero rows of $A'$. Finally define the set
\[
\AA_p(P_0)=
\left\{\ba{ll}
\emptyset &\text{ if }\yp_1X^e-\yp_c\text{ has no linear}\\
 & \text{ factors for some }c\in\{2,\ldots,d\}, \\
\left\{
\left(\ba{c}
E_0 \\ 
E_{(i_2,\ldots,i_d)}
\ea\right)
 : 1\leq i_c\leq l_c\right\} & \text{ otherwise.}
\ea\right.
\]

\begin{theo}\label{theorem2}
Suppose $P_0\in H$ and $p$ is a rational prime satisfying \textbf{(p1)$^{\text{split}}$,(p2)} and \textbf{(p3)$^{\text{ram}}$}. If $\rank_{\F_p}(\overline E)=d$ for every $E\in\AA_p(P_0)$ or if $\AA_p(P_0)=\emptyset$, then $B_p(P_0)\cap H=\{P_0\}$.
\end{theo}

\begin{proof}
Let $P\in\left(B_p(P_0)\cap H\right)\setminus\{P_0\}$. Since $P\neq P_0$ there exist integers $n_1',\ldots,n_r'$, not all zero, such that 
\[
N[P-P_0]=n_1'D_1+\ldots+n_r'D_r
\]
in $J(K)$. Let $\tau_c$ be a well-behaved uniformizer for $\C_{/\O_c}$ at $P_0$, for $c\in\{1,\ldots,d\}$. Now 
\[
n_1\int_{D_1}\om_f^c+\ldots+n_r\int_{D_r}\om_f^c=
\int_{[P-P_0]}\om_f^c=\al_c^{(f)}t_c
+\beta_c^{(f)}t_c^2
\]
where $n_q=n_q'/N$ for $q\in\{1,\ldots,r\}$. Writing this in terms of matrices we get
\[
A\bn=\bw\bt+\bw'\bt'
\]
where $\bt$ and $\bt'$ are the column vectors $\left(t_1,\ldots,t_d\right)$ and $\left(t_1^2,\ldots,t_d^2\right)$ respectively. Multiplying by the unimodular matrix $U$ and looking at the last $j$ rows we see that 
\[
\mathbf 0 = E_0\bt +\bw''\bt'.
\] 
Dividing by $p^s$ and reducing modulo $p$ we see that 
\be\la{eqthm1}
\mathbf 0 \equiv \ol{E_0}\bv\mod p,
\ee 
since $s\geq 1$. 

Also we can write $\psi(P)-\psi(P_0)$ as a powerseries in $t_c$ for every $c$. Since the ramification index of $\psi^{\p_c}$ at $P_0$ is $e$ by (p3)$^{\text{ram}}$, these powerseries will all start from the $e$-th term. Using these expansions we get the $d$ equalities 
\[
\psi(P)-\psi(P_0)=\yp_ct_c^e+\rho_ct_c^{e+1}.
\]
By Lemma \ref{lemma1} we know that there exist $i_c\in\{1,\ldots,l_c\}$ for each $c\in\{2,\ldots,d\}$ such that
\[
t_1-\hat\gamma_c^{(i_c)}t_c\equiv 0 \mod p^{s(e+1)},
\]
where $s=\ord_p(t_1)=\ldots=\ord_p(t_d)$.  Since $e>1$ we get that 
\[
t_1-\hat\gamma_c^{(i_c)}t_c\equiv 0\mod p^{s+1}.
\]
This can be re-written as 
\[
\mathbf 0 \equiv \ol{E_{(i_2,\ldots,i_d)}}\bt \mod p^{s+1}.
\]
Dividing by $p^s$ and reducing modulo $p$ we obtain that 
\be\la{eqthm2}
\mathbf 0 \equiv \ol{E_{(i_1,\ldots,i_d)}}\bv\mod p.
\ee 
Relations \eqref{eqthm1} and \eqref{eqthm2} put together imply that there exists $E\in\AA_p(P_0)$ with
\[
\mathbf 0 \equiv \ol E\bv\mod p.
\]
But this is a contradiction since $\rank_{\F_p}(\ol{E})=d$ and $\bv\not\equiv 0\mod p$.
\end{proof}

\subsubsection*{$p$ inert}

Now if we assume that $K$ is a quadratic extension of $\Q$ the following results show how we may also use an odd rational prime $p$ which is inert in $K$. This might prove useful in practice, since a split prime satisfying the properties needed to perform Chabauty, might be too big for computational purposes. In the following results we denote by $\p$ the unique prime above $p$, which has norm $p^2$, by $K_\p$ the completion of $K$ with respect to $\p$ and by $\O_\p$ the ring of integers of $K_\p$. Let $\C_{/\O_\p}$ be a minimal, regular and proper model for $C$ over $\O_\p$.

\begin{lemm}\la{lemma2}
Suppose that $[K:\Q]=2$, $P_0\in H$ and $p$ is a rational prime that satisfies \textbf{(p1)$^{\text{inert}}$,(p2)} and \textbf{(p3)$^{\text{ram}}$}. Let $\tau$ be a well-behaved uniformizer of $\C_{\p}$ at $P_0$ and denote by
\[
\yp T^e+\sum_{i=1}^{\infty}\rho_iT^{e+i}\in K_\p[[T]]
\]
the formal expansion of $\psi^\p-\psi(P_0)$ in terms of $\tau$. Write
\[
\yp T^e=(v_1\th_1 +v_2\th_2 )(T_1\th_1 +T_2\th_2 )^e =W_1(T_1,T_2)\th_1 +W_2(T_1,T_2)\th_2 ,
\]
where $W_1,W_2\in\Z_p[T_1,T_2]$ are quadratic forms of degree $e$. Suppose further that $\ord_p(\Delta)=0$, where $\Delta$ is the discriminant of $W_2$. Then if there exists $P\in\left(B_p(P_0)\cap H\right)\setminus\{P_0\}$, $W_2$ has a linear factor $g_1T_1-g_2 T_2$ satisfying
\[
g_1t_1-g_2 t_2\equiv 0\mod p^{s(e+1)},
\]
where $t_1, t_2$ are defined by $\tau(P)=t_1\th_1 +t_2\th_2 $ and $1\leq s=\min\left(\ord_p(t_1),\ord_p(t_2)\right)<\infty$.
\end{lemm}

\begin{proof}
By substituting $P$ in the powerseries expansion of $\psi^\p-\psi(P_0)$ we get
\begin{eqnarray*}
\psi^\p-\psi(P_0)=\yp \tau(P)^e+\rho \tau(P)^{e+1} &=& \left(W_1(t_1,t_2)\th_1 +W_2(t_1,t_2)\th_2 \right)+ \\ 
& & \left(W_3(t_1,t_2)\th_1 +W_4(t_1,t_2)\th_2 \right),
\end{eqnarray*}
where $\yp\in\O_\p^*,\rho\in\O_\p$ and $W_3,W_4$ are obtained from $\rho(t_1\th_1 +t_2\th_2 )^{e+1}$ as $W_1$ and $W_2$ were obtained from $\yp$ in the statement of the lemma. Since $\psi(P)\in \Q$, we have that
\[
W_2(t_1,t_2)=-W_4(t_1,t_2).
\]
Furthermore since $P\neq P_0$ either $t_1\neq 0$ or $t_2\neq 0$, so $s:=\min\left(\ord_p(t_1),\ord_p(t_2)\right)$ is finite. Combining this with the fact that $\ord_\p(\tau(P))\geq 1$ we can see that $s$ is actually a positive integer. We have that
\[
W_2(t_1,t_2)\equiv 0 \mod p^{s(e+1)}.
\]
By Hensel's Lemma (since $\ord_p(\Delta)=0$) we have that there exist $g_1,g_2\in\Z_p$ such that
\[
(g_1T_1-g_2T_2)\mid W_2(T_1,T_2)
\]
and
\[ 
g_1t_1-g_2t_2\equiv  0 \mod p^{s(e+1)}.
\]
\end{proof}

Let $K,P_0,e,p,t_1,t_2,W_2$ be as in the statement of Lemma \ref{lemma2} above. Suppose that $(g_1^{(1)}T_1-g_2^{(1)}T_2),\ldots,(g_1^{(l)}T_1-g_2^{(l)}T_2)$ are all the linear factors of $W_2(T_1,T_2)$. Define the $l$ matrices $E_{(i)}\in M_{1,2}(\Z_p)$ by
\[
E_{(i)}=\left(\ba{cc} g_1^{(i)} & -g_2^{(i)} \ea\right)
\] 
for $i\in\{1,\ldots,l\}$. 

Let $\{\om_f\}_{1\leq f\leq g}$ be a basis of $\Omega_{\C/\O_\p}$, $\{\th_1 ,\th_2 \}$ be an integral basis of $\O_\p$ over $\Z_p$, and $L_0=\langle D_1,\ldots,D_r\rangle$ be a subgroup of $J(K)$ of index $N\in\Z_{>0}$. Define $\{A^{(1)},\ldots,A^{(g)}\}\subseteq M_{2,r}(\Q_p)$ to be the matrices whose entries are defined by 
\[
A_{1,q}^{(f)}\th_1 +A_{2,q}^{(f)}\th_2 =\int_{D_q}\om_f,
\] 
and for $f\in\{1,\ldots,g\}$ define $\bw^{(f)}\in M_{2,2}(\Z_p)$ to be the matrix representing in coordinates $\al^{(f)}$, the coefficient of the linear term in the (formal) powerseries expansion of $\int_{[P-P_0]}\om_f$ in terms of the uniformizer $\tau$ of $\C_{/\O_\p}$ at $P_0$. Now let $A$ be the $2g\times r$ matrix 
\[
A=
\left(\ba{c}
A^{(1)} \\ 
\vdots \\ 
A^{(g)}
\ea\right)
\] 
with entries in $\Q_p$ and $\bw$ be the $2g\times 2$ matrix
\[
\bw=
\left(\ba{c}
\bw^{(1)} \\ 
\vdots \\ 
\bw^{(g)}
\ea\right)
\]
with entries in $\Z_p$. Let $h$ be the smallest integer such that $p^hA$ has entries in $\Z_p$. Again let $U\in \GL_{dg}(\Z_p)$ be the unimodular matrix such that
\[
A'=U(p^hA)
\]
is in Hermite Normal Form. We denote by $E_0\in M_{j,d}(\Z_p)$ the matrix formed by the last $j$ rows of $U\bw$, where $j$ is the number of zero rows of $A'$. Finally define the set
\[
\AA_p(P_0)=
\left\{\ba{ll}
\emptyset &,\text{ if }W(T_1,T_2)\text{ has no linear factors,}\\ 
\left\{
\left(\ba{c}
E_0 \\ 
E_{(i)}
\ea\right) 
: 1\leq i\leq l\right\} &\text{, otherwise.}
\ea\right.
\]

\begin{theo}\label{theorem3}
Suppose $[K:\Q]=2$, $P_0\in H$ and $p$ is a rational prime satisfying \textbf{(p1)$^{\text{inert}}$,(p2),(p3)$^{\text{ram}}$} and that $W_2,\Delta$ are defined as in the statement of Lemma \ref{lemma2} with $\ord_p\left(\Delta\right)=0$. Then if $\rank_{\F_p}(\overline E) = 2$ for every $E\in\AA_p(P_0)$ or if $\AA_p(P_0)=\emptyset$, we have that $B_p(P_0)\cap H=\{P_0\}$.
\end{theo}

\begin{proof}
Let $P\in\left(B_p(P_0)\cap H\right)\setminus\{P_0\}$. Since $P\neq P_0$ there exist integers $n_1',\ldots,n_r'$, not all zero, such that 
\[
N[P-P_0]=n_1'D_1+\ldots+n_r'D_r
\]
in $J(K)$. Let $\tau$ be a well-behaved uniformizer for $\C_{/\O_\p}$ at $P_0$. Now 
\[
n_1\int_{D_1}\om_f+\ldots+n_r\int_{D_r}\om_f =\int_{[P-P_0]}\om_f =\al^{(f)}\tau(P)+\beta^{(f)}\tau(P)^2
\]
where $n_q=n_q'/N$ for $q\in\{1,\ldots,r\}$. Writing this in terms of matrices we get
\[
A\bn=\bw\bt+\bw'\bt'
\]
where $\bt$ and $\bt'$ are the column vectors $\bt=\left(t_1,t_2\right)$ and $\bt'=\left(t_1',t_2'\right)$ with $t_1',t_2'$ defined by $\tau(P)^2=t_1'\th_1 +t_2'\th_2 $. Multiplying by the unimodular matrix $U$ and looking at the last $j$ rows we see that 
\[
\mathbf 0 = E_0\bt +\bw''\bt'.
\] 
Dividing by $p^s$ and reducing modulo $p$ we see that \be\la{eqthm3}
\mathbf 0 \equiv \ol{E_0}\bv\mod p,
\ee 
since $s\geq 1$. 

Also we can write $\psi(P)-\psi(P_0)$ as a powerseries in $\tau(P)$. Since the ramification index of $\psi^{\p}$ at $P_0$ is $e$ by (p3)$^{\text{ram}}$, this powerseries will start from the $e$-th term. By Lemma \ref{lemma2} we know that there exists $i\in\{1,\ldots,l\}$ such that
\[
g_1^{(i)}t_1-g_2^{(i)}t_2\equiv 0 \mod p^{s(e+1)},
\] 
where $s=\min\left(\ord_p(t_1),\ord_p(t_2)\right)$.  Since $e>1$ we get that 
\[
g_1^{(i)}t_1-g_2^{(i)}t_2\equiv 0\mod p^{s+1}.
\]
This can be re-written as 
\[
\mathbf 0 \equiv \ol{E_{(i)}}\bt \mod p^{s+1}.
\]
Dividing by $p^s$ and reducing modulo $p$ we obtain that \be\la{eqthm4}
\mathbf 0 \equiv \ol{E_{(i)}}\bv\mod p.
\ee 
Relations \eqref{eqthm3} and \eqref{eqthm4} put together imply that there exists $E\in\AA_p(P_0)$ with
\[
\mathbf 0 \equiv \ol E\bv\mod p.
\]
But this is a contradiction since $\rank_{\F_p}(\ol{E})=2$ and $\bv\not\equiv 0\mod p$.
\end{proof}

\subsection{Applying Chabauty}
\begin{exam}\label{example1}
Let $K$ be the number field defined by $\Q[x]/(x^2-x+3)$, and denote by $\th$ the corresponding image of $x$ in the quotient. Consider the three hyperelliptic curves $C_1$, $C_2$ and $C_3$ defined over $K$ by the equations 
\begin{align}
{C_1}_{/K}: & \quad y^2=x^5+\th  x^4-x^3+x^2+(\th -1)x-1\label{curve1}\\
{C_2}_{/K}: & \quad y^2=-x^5-\th  x^4+x^3-x^2-(\th -1)x+1\label{curve2}\\
{C_3}_{/K}: & \quad y^2=(-\th  + 5)x^5 + (4\th  + 3)x^4 + (\th  - 5)x^3 + (-\th  + 5)x^2 + (5\th  - 2)x + \th  - 5\label{curve3}
\end{align}
and the ``$x$-coordinate'' maps $\psi_1,\psi_2,\psi_3$  from the projective models of these curves to the projective line
\[
\psi_i:C_i\rightarrow\P^1,\quad \psi_i\left(X,Y,Z\right) = \left(X,Z\right).
\]
Denote $C_i(K)\cap \psi_i^{-1}\left(\P^1(\Q)\right)$ by $H_i$ for $1\leq i \leq 3$. Let
\begin{equation}\label{hfake}
\begin{array}{ll}
H_1':= & \{(-1,-1,1),(-1,1,1),(1,0,0)\}  \subseteq H_1 \\
H_2':= & \{(0,1,1),(0,-1,1),(1,0,0)\}  \subseteq H_2\\
H_3':= & \{(1,0,0)\}  \subseteq H_3.
\end{array}
\end{equation}
After searching for $K$-rational points on $C_1$, $C_2$ and $C_3$ it appears that actually $H_1'=H_1, H_2'=H_2$ and $H_3'=H_3$. The first step towards proving this is using Theorems \ref{firsttheo},\ref{theorem2} and \ref{theorem3} together with the relevant information (computed using MAGMA \cite{MR1484478}) presented in the following tables:

In \textsc{Table} \ref{tableranks} we observe that the rank of the Mordell-Weil group of the Jacobian variety of $C_1$ is equal to $3>d(g-1)=2$, making it impossible to use the classical method of Chabauty which requires that $r\leq d(g-1)$. See for example \cite{siksek}. Our method is applicable in cases where the rank $r$ of $J(K)$ satisfies
\[
r\leq dg-1.
\]

In \textsc{Table} \ref{tab:tableperiods} we give the matrix $A$ with entries in $\Q_p$ and a corresponding unimodular matrix $U$ with entries in $\Z_p$, such that $UA$ is a matrix in Hermite Normal Form.

\begin{table}[h]
\begin{tiny}
\begin{center}
\begin{tabular}{|@{}c@{}|@{}c@{}|@{}c@{}|@{}c@{}|@{}c@{}|}
\hline
 $C$ & $\rank_{\F_2}\left(\Sel^{(2)}(J_C/K)\right)$ & $\ba{c}\text{lin. ind.} \\ \text{non-torsion divisors}\\ \text{(In Mumford Representation)}\ea$ & $\rank\left(J_C(K)\right)$ \\ \hline
$C_1$ & $3$ & $\ba{l}(x^2 - x + 1, -x + 1)\\ (x^2 + (\th  - 1)x - 1, (\th  - 1)x - 1)\\ (x^2 + (-\th  - 1)x - \th  + 2, (3\th  - 1)x + \th  - 4)\ea$ & $ 3$ \\ \hline
$C_2$ & $1$ & $\quad (x^2 - x - \th , (\th  - 2)x - 2)\qquad\qquad\qquad\qquad\ \ $ & $ 1$ \\ \hline
$C_3$ & $1$ & $\quad (x^2 + (2\th - 1)x + \th - 3, (-4\th - 3)x - 5\th + 2)\quad$ & $1$ \\ \hline
\end{tabular}
\end{center}
\end{tiny}
\caption{The Mordell-Weil data for $C_1$, $C_2$ and $C_3$.}\label{tableranks}
\end{table}

\begin{table}[h]
\begin{tiny}
\begin{center}
\begin{tabular}{|c|c|@{}c@{}|@{}c@{}|}
\hline  & $p$ & $A$ & $U$ \\ 
\hline $C_1$ & 89 & $\left(\begin{array}{ccc}
 70 & 82 & 51 \\
 70 &  61 & 86 \\
 55 & 3 & 58 \\
 29 & 38 & 28  
\end{array}\right)\times 89+O(89^2)$ & $\left( \begin{array}{cccc}
 -6 & 1 & -13 & 3 \\
 -5 & -1 & 5 & 5 \\
 -42 & 0 & 14 & 38 \\
 6 & 2 & -6 & -11 
\end{array} \right)+O(89)$ \\ 
\hline $C_2$ & $23$ & $\left(\begin{array}{c}
 -6 \\
 -11 \\
 -11 \\
 -11
\end{array}\right)\times 23+O(23^2)$ & $\left( \begin{array}{cccc}
-2 & 0 & 1 & 0 \\
1 & 0 & 11 & 1 \\
-11 & 0 & 6 & 0 \\
0 & 1 & -1 & 0 
\end{array} \right)+O(23)$ \\
\hline
$C_3$ & $71$ & 
$\left(\ba{c}
58  \\
60  \\
47  \\
48
\ea\right)\times 71+O(71^2)$ & 
$\left(\ba{cccc}
0 & 0 & -1 & 1 \\
1 & 0 & -13 & 13 \\
0 & 1 & -11 & 11\\
0 & 0 & 23 & -24
\ea\right)+O(71)$ \\
\hline 
\end{tabular} 
\end{center}
\end{tiny}
\caption{The period matrices for $C_1,C_2$ and $C_3$.}\label{tab:tableperiods}
\end{table}

\begin{rema}
Since, in practice, we are always working with finite precision, the matrix $U$ presented in \textsc{Table} ~\ref{tab:tableperiods} is not unique. This does not however affect the ranks of the matrices $\ol E$ for $E \in \AA_{p}(P_0)$.
\end{rema}

\begin{table}[h]
\begin{tiny}
\begin{center}
\begin{tabular}{|@{}c@{}|@{}c@{}|@{}c@{}|@{}c@{}|@{}c@{}|}
\hline $P_0$ & $\tau$  & $\bw$ & $\{E_{j_2}\}$ & $\AA_p(P_0)$ \\ 
\hline $\ba{c}(1,1,-1)\\ \text{unramified}\ea$ & $x+1$ & $\left(\begin{array}{c}
1 \\
-1 \\
1 \\ 
-1
\end{array} \right)+O(89)$ & N/A & $\left(\begin{array}{c}
9
\end{array} \right)+O(89)$ \\ 
\hline $\ba{c}(1,-1,-1)\\ \text{unramified}\ea$ & $x+1$ & $\left(\begin{array}{c}
-1 \\ 
1 \\
-1 \\ 
1
\end{array} \right)+O(89)$ & N/A & $\left(\begin{array}{c}
-9
\end{array} \right)+O(89)$ \\ 
\hline $\ba{c}(1,0,0)\\ \text{ramified}\ea$ & $\frac{(\th  - 1)x^2}{2y}$ & $\left(\begin{array}{cc}
0 & 0 \\
41 & 0 \\  
0 & 0 \\ 
0 & 79 
\end{array} \right)+O(89)$ & $\ba{c}\left(\begin{array}{cc}
1 & 41
\end{array} \right)+O(89),\\ \left(\begin{array}{ll}
1 & -41
\end{array} \right)+O(89)\ea$ & $\ba{c}\left(\begin{array}{cc}
82 & 21 \\ 
1 & 41 
\end{array} \right)+O(89),\\ \left(\begin{array}{cc}
82 & 21 \\ 
1 & -41
\end{array} \right)+O(89)\ea$ \\ 
\hline 
\end{tabular}
\end{center}
\end{tiny}
\caption{Chabauty data for $C_1$}\label{tab:tab1}
\end{table}
\begin{table}[h]
\begin{tiny}
\begin{center}
\begin{tabular}{|@{}c@{}|@{}c@{}|@{}c@{}|@{}c@{}|@{}c@{}|}
\hline $P_0$ & $\tau$  & $\bw$ & $\{E_{j_2}\}$ & $\AA_p(P_0)$ \\ 
\hline $\ba{c}(0,-1,1)\\ \text{unramified}\ea$ & $x$ & $\left(\begin{array}{c}
-1 \\
0 \\ 
-1 \\  
0
\end{array} \right)+O(23)$ & N/A & $\left(\begin{array}{c}
11\\
5\\
1
\end{array} \right)+O(23)$ \\ 
\hline $\ba{c}(0,1,1)\\ \text{unramified}\ea$ & $x$ & $\left(\begin{array}{c}
1 \\
0 \\ 
1 \\  
0
\end{array} \right)+O(23)$ & N/A & $\left(\begin{array}{c}
-11\\
-5\\
-1
\end{array} \right)+O(23)$ \\ 
\hline $\ba{c}(1,0,0)\\ \text{ramified}\ea$ & $-\frac{\th +2}{6}\frac{x^2}{y}$ & $\left(\begin{array}{cc}
0 & 0 \\
17 & 0 \\ 
0 & 0 \\  
0 & 5
\end{array} \right)+O(23)$ & $\ba{c}\left(\begin{array}{cc}
1 & 3
\end{array} \right)+O(23),\\ \left(\begin{array}{cc}
1 & -3
\end{array} \right)+O(23)\ea$ & $\ba{c}\left(\begin{array}{cc}
0 & 5\\
0 & 0\\
-6 & 0\\
1 & 3 
\end{array} \right)+O(23),\\ \left(\begin{array}{cc}
0 & 5\\
0 & 0\\
-6 & 0\\
1 & -3
\end{array} \right)+O(23)\ea$ \\ 
\hline 
\end{tabular}
\end{center}
\end{tiny}
\caption{Chabauty data for $C_2$}\label{tab:tab2}
\end{table}
\begin{table}[h]
\begin{tiny}
\begin{center}
\begin{tabular}{|c|c|c|c|c|}
\hline $P_0$ & $\tau$  & $\bw$ & $\{E_{j_2}\}$ & $\AA_p(P_0)$ \\ 
\hline $\ba{c}(1,0,0)\\ \text{ramified}\ea$ & $\frac{2(5-\th)x^2}{y}$ & $\left(\begin{array}{cc}
0 & 0 \\
56 & 0 \\ 
0 & 0 \\  
0 & 64
\end{array} \right)+O(71)$ & $\emptyset$ & $\emptyset$ \\ 
\hline 
\end{tabular}
\end{center}
\end{tiny}
\caption{Chabauty data for $C_3$}\label{tab:tab3}
\end{table}
Using Theorems \ref{firsttheo}, \ref{theorem2} and the data in \textsc{Tables} ~\ref{tab:tab1},~\ref{tab:tab2} and ~\ref{tab:tab3} we deduce the following:
\ben{a}
\item $B_{89}(P_0)\cap H_1 =\{P_0\}$ for every $P_0\in H_1'$
\item $B_{23}(P_0)\cap H_2 =\{P_0\}$ for every $P_0\in H_2'$
\item $B_{71}(P_0)\cap H_3 =\{P_0\}$ for every $P_0\in H_3'$
\een
\end{exam}

\section{Mordell-Weil sieve}\la{section4}

\begin{equation}\label{cdmw}
\xymatrix{
H\ \ \subseteq\ \ C(K) \ar@<2.7ex>[dd]^{\red_{p_i}} \ar@<-6.3ex>[dd]_{\red_{p_i}}\ar[rr]^{\qquad\iota} & & J(K)                                      \ar[dd]^{\red_{p_i}}\\
& & \\
\G_i \subseteq\ \prod\limits_{\p\mid p_i} C\left(k_\p\right) \ar@<1.6ex>[rr]^{\quad\iota} & & \prod\limits_{\p\mid p_i}J\left(k_\p\right)
}
\end{equation}

\begin{theo}\label{theoremmw}
Let $L=\langle D_1,\ldots , D_r\rangle <J(K)$ be a subgroup of the Mordell-Weil group of finite index equal to $N$ and $p_0,p_1,\ldots,p_b$ be rational primes satisfying
\ben{i}
\item $p_i$ does not ramify in $\O_K$ for $0\leq i\leq b$.
\item $C_{/K}$ has good reduction for every prime $\p\mid p_i$, for $0\leq i\leq b$.
\item $\#\prod\limits_{\p\mid p_i}J(k_{\p})$ is coprime with $N$ for $0\leq i\leq b$.
\een
Let 
\[
\G_i=\left\{\bP\in \prod\limits_{\p\mid p_i}C(k_{\p}) : \ol{\psi^\p}(\bP_\p)=\psi_{\q}(\bP_\q)\in\P^1(\F_{p_i})\quad\forall \p,\q\mid p_i\right\}.
\]
Let $L_0:=L\cap\Kernel\left(\red_{p_0}\right)$ and define inductively $L_i:=L_{i-1}\cap\Kernel\left(\red_{p_i}\right)$ for $1\leq i\leq b$. Then for every $\bP\in\G_0$ define $W_{0,\bP}=\{l\in L/L_0 : \red_{p_0}(w)=\iota(\bP)\}$ and then inductively $W_{i,\bP}:=\{w+l : w\in W_{i-1,\bP},l\in L_{i-1}/L_i,\red_{p_i}(w+l)\in\iota\left(\G_i\right)\}$ for $1\leq i\leq b$. Then if $W_{b,\bP}=\emptyset$ we have that $\B_{p_0}(\bP)\cap H=\emptyset$.
\end{theo}

\begin{proof}
Suppose there exists some point $P$ in $\B_{p_0}(\bP)\cap H$. Then $N\iota(P)=n_1D_1+\ldots +n_rD_r$ for some $n_1,\ldots ,n_r\in \Z$. Since condition $(iii)$ holds for $p_0$ we have that $\red_{p_0}\left(\iota(P)\right)\in\red_{p_0}\left(L\right)$ and also $\red_{p_0}\left(\iota(P)\right)=\iota(\bP)$ by commutativity of diagram \eqref{cdmw}, in other words we can find $w_{0,P}\in L/L_0$ such that $\red_{p_0}\left(w_{0,P}\right)=\red_{p_0}\left(\iota(P)\right)$. In particular $w_{0,P}\in W_{0,\bP}$. Now suppose that for $i=0,\ldots , i-1$ we have $w_{i-1,P}\in W_{i-1,\bP}$ such that $\red_{p_{i-1}}\left(w_{i-1,P}\right)=\red_{p_{i-1}}\left(\iota(P)\right)$. We now have
\[
\iota(P)-w_{i-1,P}  \in  \bigcap\limits_{j=0}^{i-1}\Kernel\left(\red_{p_j}\right).
\]
If we multiply by the index $N$ we have
\[
N\left(\iota(P)-w_{i-1,P}\right)  \in  L\cap\left(\bigcap\limits_{j=0}^{i-1}\Kernel\left(\red_{p_j}\right)\right)=L_{i-1}
\]
and if we reduce both sides modulo $p_i$ we get
\[
N\red_{p_i}\left(\iota(P)-w_{i-1,P}\right)  \in  \red_{p_i}\left(L_{i-1}\right).
\]
But since $N$ is coprime with $\# \red_{p_i}\left(L_{i-1}\right)$ by $(iii)$ this implies that
\[
\red_{p_i}\left(\iota(P)-w_{i-1,P}\right) \in \red_{p_i}\left(L_{i-1}\right).
\]
So there exists $l\in L_{i-1}/L_i$ such that $\red_{p_i}(l)=\red_{p_i}\left(\iota(P)-w_{i-1,P}\right)$. But we can now define an element of $W_{i,\bP}$ by
\[
w_{i,P}:=w_{i-1,P}+l\in W_{i,\bP}.
\]
In particular $W_{b,\bP}$ is non-empty.
\end{proof}

\begin{exam}\label{example2}
Let $C_1,C_2$ and $C_3$ be the curves defined in Section \ref{chabsection} \eqref{curve1},\eqref{curve2} and \eqref{curve3}. Then 
\ben{a}
\item $\B_{89}(\bP)\cap H_1 =\emptyset$ for every $\bP\in \prod\limits_{\p\mid 89} C_1(k_\p)\setminus\red_{89}\left(H_1'\right)$. This was shown after taking $\{p_0,p_1,p_2,p_3=p_b\}=\{89,673,859,131\}$ and using Theorem \ref{theoremmw} after checking that conditions $(i),(ii)$ and $(iii)$ were satisfied for each of these primes.
\item $\B_{23}(\bP)\cap H_2 =\emptyset$ for every $\bP\in \prod\limits_{\p\mid 23} C_2(k_\p)\setminus\red_{23}\left(H_2'\right)$. The primes used here were $\{23,43\}$.
\item $\B_{71}(\bP)\cap H_3 =\emptyset$ for every $\bP\in \prod\limits_{\p\mid 71} C_3(k_\p)\setminus\red_{71}\left(H_3'\right)$. The primes used were $\{71,131\}$.
\een
\end{exam}

\begin{lemm}\la{lemmahhh}
Let $C_1,C_2$ and $C_3$ be the curves defined in \eqref{curve1}, \eqref{curve2} and \eqref{curve3} respectively and let $\psi_i$ for $1\leq i\leq 3$ be the corresponding ``$x$-coordinate'' maps from the curves to the projective line. We have that $H_i=C_i(K)\cap \psi_i^{-1}\left(\P^1(\Q)\right)=H_i'$ for $1\leq i\leq 3$, where the $H_i'$ are as in \eqref{hfake}.
\end{lemm}

\begin{proof}
Just note that the corresponding parts of Examples \ref{example1} and \ref{example2} together give the required result.
\end{proof}

\section{Applications to Diophantine Problems}\la{section5}

In this section we consider an example of a curve $Y$ defined over $\Q$ whose set of rational points is computed using the methods presented in the previous sections. This illustrates how all of the existing methods (\cite{MR2011330},\cite{MR0011005},\cite{MR808103},\cite{siksek}) may fail due to theoretical or computational restrictions, while the  methods in this paper remain applicable. The usefulness of this technique should be more apparent when used on curves that are not hyperelliptic, for example more general cyclic covers of the projective line. These might have Jacobians of Mordell-Weil rank large enough to pose a theoretical obstruction to the use of classical Chabauty or its refinement in \cite{siksek} and also fail to be related to collections of curves of genus $1$, where ``Elliptic Curve Chabauty'' might be applicable.

\subsection{The Equation $y^{2}=(x^3+x^2-1)\Phi_{11}(x)$}

Let $Y$ be the genus $6$ hyperelliptic curve defined by the equation

\begin{eqnarray*}
y^2 & = & (x^3+x^2-1)\Phi_{11}(x)\\
    & = & x^{13} + 2x^{12} + 2x^{11} + x^{10} + x^9 + x^8 + x^7 + x^6 + x^5 + x^4 + x^3 - x - 1.
\end{eqnarray*}
We prove that \[Y(\Q)=\{\infty\},\] using the techniques developed in the previous sections.

We start by noticing that over $K=\Q[x]/(x^2-x+3)=\Q(\th )$, 
\[
(x^3+x^2-1)\Phi_{11}(x)=(x^3+x^2-1)f(x)g(x)
\]
where 
\begin{align*}
f(x)= & x^5 + \th  x^4 - x^3 + x^2 + (\th  - 1) x - 1\\
g(x)= & x^5 + (-\th  + 1)x^4 - x^3 + x^2 - \th  x - 1.
\end{align*}
Consider the four fibred products of curves defined over $K$ by 
\begin{eqnarray*} 
& D_1: & 
\left\{\ba{l}
y_1^2=x^3+x^2-1 \\ 
y_2^2=f(x)\\ 
y_3^2=g(x)
\ea\right.
\\
& D_2: & 
\left\{\ba{l} 
y_1^2=x^3+x^2-1 \\ 
y_2^2=-f(x)\\ 
y_3^2=-g(x)
\ea\right.
\\
& D_3: & 
\left\{\ba{l} 
y_1^2=23(x^3+x^2-1) \\ 
y_2^2=(5-\th)f(x)\\ 
y_3^2=(\th +4)g(x)
\ea\right.
\\
& D_4: & 
\left\{\ba{l} 
y_1^2=23(x^3+x^2-1) \\ 
y_2^2=-(5-\th)f(x)\\ 
y_3^2=-(\th +4)g(x)
\ea\right.
\end{eqnarray*}
and the corresponding covering maps 
\begin{eqnarray*} 
 & \delta_i:D_i\rightarrow Y, & \quad \delta_i(x,y_1,y_2,y_3)=
 \left\{\ba{ll}
 (x,y_1y_2y_3) & \text{for } i=1,2\\
 \left(x,\frac{1}{23}y_1y_2y_3\right) & \text{for } i=3,4.
\ea\right.
\end{eqnarray*}
For each of these we have another covering map
\[
\gamma_i:D_i\rightarrow C_i,\qquad\gamma_i(x,y_1,y_2,y_3)=(x,y_2),
\]
where $C_1$, $C_2$ and $C_3$ are the genus $2$ hyperelliptic curves defined in Section \ref{chabsection} \eqref{curve1}, \eqref{curve2} and \eqref{curve3} and $C_4$ is the genus $2$ hyperelliptic curve defined over $K$ by
\[
y^2=-(5-\th)f(x).
\]

\begin{lemm}\label{lemmaphi11}
Let $H_i:=C_i(K)\cap \psi_i^{-1}\left(\P^1(\Q)\right)$ for $1\leq i \leq 4$. We have that 
\[
Y(\Q)=\bigcup_{i=1}^4 \delta_i\left(\gamma_i^{-1}(H_i)\right).\]
\end{lemm}

\begin{proof}
Define the map $\mu$
\[
\mu:Y(\Q)\rightarrow \left(\Q^*/\Q^{*2}\right)\times\left(K^*/K^{*2}\right)
\]
\[
\mu(x,y)=
\left\{\ba{ll}
\left((x^3+x^2-1)\Q^{*2},f(x)K^{*2}\right) &\text{, if }(x,y)\neq \infty \\
\left((1)\Q^{*2},(1)K^{*2}\right) &\text{, if }(x,y)=\infty
\ea\right.
\]
The image of this map is contained in 
\[
\Kernel\left(\ol N\right)\cap \left(\Q(2,S_1)\times K(2,S_2)\right),
\]
where $S_1=\Supp(\Resa(x^3+x^2-1,f(x)g(x)))=\{23\}$, $S_2=\Supp(\Resa(f(x),(x^3+x^2-1)g(x)))=\{\p\}$, with $\p$ one of the primes of $\O_{K}$ above $23$, and
\[
\ol N:\left(\Q^*/\Q^{*2}\right)\times\left(K^*/K^{*2}\right)\rightarrow \Q^*/\Q^{*2}
\]
the reduction of the product of norm maps
\[
N:\Q\times K\rightarrow\Q,\]\[N(h_1,h_2)=h_1N_{K/\Q}(h_2).
\]
Also for each element $(\al_1,\al_2)$ in the image of $\mu$ we can associate a cover $\delta_{(\al_1,\al_2)}:D_{(\al_1,\al_2)}\rightarrow Y$ defined by 
\[
D_{(\al_1,\al_2)}:  
\left\{\ba{ll}
y_1^2=\al_1(x^3+x^2-1) \\
y_2^2=\al_2 f(x)\\ 
y_3^2=\sigma(\al_2)g(x)
\ea\right.
\text{ and}
\]
\[
\delta_{(\al_1,\al_2)}(x,y_1,y_2,y_3)=(x,\nu y_1y_2y_3),
\]
where $\nu$ is a rational number satisfying $\nu^2\al_1N_{K/\Q}(\al_2)=1$. We now have
\[
Y(\Q)=\bigcup_{(\al_1,\al_2)\in\Image(\mu)}\delta_{(\al_1,\al_2)}\left(H_{(\al_1,\al_2)}'\right),
\]
where 
\[
H_{(\al_1,\al_2)}':= D_{(\al_1,\al_2)}(K)\cap \psi_{(\al_1,\al_2)}^{-1}\left(\P^1(\Q)\right),
\]
with $\psi_{(\al_1,\al_2)}:D_{(\al_1,\al_2)}\rightarrow\P^1$ and $\psi_{(\al_1,\al_2)}(x,y_1,y_2,y_3)=(x,1)$.

A computation gives that 
\[
\Kernel\left(\ol N\right)\cap \left(\Q(2,S_1)\times K(2,S_2)\right)=\left\{(1,1),(1,-1),(23,5-\th),(23,\th-5)\right\},
\]
so we only need to be concerned with the covers $D_1=D_{(1,1)}$, $D_2=D_{(1,-1)}$, $D_3=D_{(23,5-\th)}$ and $D_4=D_{(23,\th -5)}$. Finally it is obvious that 
\begin{align*}
H_{(1,1)}'&=\gamma_1^{-1}(H_1)\\
H_{(1,-1)}'&=\gamma_2^{-1}(H_2)\\
H_{(23,5-\th)}'&=\gamma_3^{-1}(H_3)\\
H_{(23,\th-5)}'&=\gamma_4^{-1}(H_4).
\end{align*}
\end{proof}

We can now prove the following

\begin{theo}
The only $\Q$-rational point on the curve $Y$ defined by the equation
\[
y^2=(x^3+x^2-1)\Phi_{11}(x)
\]
is the point at infinity.
\end{theo}

\begin{proof}
We have that
\[
Y(\Q)=\bigcup_{i=1}^4 \delta_i\left(\gamma_i^{-1}(H_i)\right).
\]
from Lemma \ref{lemmaphi11} and that $H_1=H_1'$, $H_2=H_2'$ and $H_3=H_3'$ from Lemma \ref{lemmahhh}. Also a $2$-Selmer group computation  shows that $J_4(K)=\{0\}$, where $J_4$ is the Jacobian variety of $C_4$ and thus $H_4=C_4(K)=\{(1,0,0)\}$. So putting these together we get that
\begin{align*}
Y(\Q) =& \delta_1\left(\gamma_1^{-1}\left(\{(-1,-1,1),(-1,1,1),(1,0,0) \}\right)\right)\\
&\cup\delta_2\left(\gamma_2^{-1}\left(\{(0,1,1),(0,-1,1),(1,0,0) \}\right)\right)\\
&\cup\delta_3\left(\gamma_3^{-1}\left(\{(1,0,0) \}\right)\right)\\
&\cup\delta_4\left(\gamma_4^{-1}\left(\{(1,0,0) \}\right)\right)\\
      =& \{\infty\} \cup \{\infty\}\cup\{\infty\} \cup \{\infty\}\\
      =& \{\infty\}.
\end{align*}
\end{proof}


\bibliographystyle{siam}
\bibliography{chabauty}

\end{document}